\documentclass[12pt]{amsart}

\usepackage{amsmath, amssymb}
\usepackage{amscd}
\usepackage{verbatim}
\usepackage{xcolor}
\newtheorem{definition}{Definition}[section]
\newtheorem{theorem}[definition]{Theorem}
\newtheorem{lemma}[definition]{Lemma}
\newtheorem{corollary}[definition]{Corollary}
\newtheorem{proposition}[definition]{Proposition}
\theoremstyle{definition}
\newtheorem{remark}[definition]{Remark}

\newtheorem{example}[definition]{Example}

\newcommand{\la}{\left\langle}
\newcommand{\ra}{\right\rangle}

\newcommand{\clB}{\mathcal{B}}

\newcommand{\bbN}{\mathbb{N}}
\newcommand{\bbC}{\mathbb{C}}

\newcommand{\ucpa}{\text{UCP}_A(H)}
\newcommand{\ucpan}{\text{UCP}_A(\bbC^n)}
\newcommand{\ucpann}{\text{UCP}_{M_n(A)}(\bbC^n)}
\newcommand{\ucpag}{\text{UCP}^G_A(H)}
\newcommand{\ucpanng}{\text{UCP}^G_{M_n(A)}(\bbC^n)}
\newcommand{\ucpang}{\text{UCP}^G_A(\bbC^n)}
\newcommand{\cpa}{\text{CP}_A(H)}

\newcommand{\cpag}{\text{CP}^G_A(H)}
\newcommand{\racom}{\rho(A)^\prime}

\newcommand{\diX}{\int^\oplus_{X}}

\newcommand{\acom}{{\rho(A)}^{\prime}}

\begin{document}


\title[Ergodic Decomposition in the Space of UCP Maps]{Ergodic Decomposition in the Space of Unital Completely Positive Maps}

\author[A.~Bhattacharya]{Angshuman Bhattacharya}
\address{Department of Mathematics, IISER Bhopal, MP 462066, India}
\email{angshu@iiserb.ac.in}

\author[C. J.~Kulkarni]{Chaitanya J. Kulkarni}
\address{Department of Mathematics, IISER Bhopal, MP 462066, India}
\address{Department of Mathematical Sciences, IISER Mohali, Punjab 140306, India}
\email{chaitanyakulkarni58@gmail.com, chaitanyakulkarni@iisermohali.ac.in}



\begin{abstract}
The classical decomposition theory for states on a C*-algebra that are invariant under a group action has been studied by using the theory of orthogonal measures on the state space \cite{BR1}. 

In \cite{BK3}, we introduced the notion of \textit{generalized orthogonal measures} on the space of unital completely positive (UCP) maps from a C*-algebra $A$ into $B(H)$. In this article, we consider a group $G$ that acts on a C*-algebra $A$, and the collection of $G$-invariant UCP maps from $A$ into $B(H)$. This article examines a $G$-invariant decomposition of UCP maps by using the theory of generalized orthogonal measures on the space of UCP maps, developed in \cite{BK3}. Further, the set of all $G$-invariant UCP maps is a compact and convex subset of a topological vector space. Hence, by characterizing the extreme points of this set, we complete the picture of barycentric decomposition in the space of $G$-invariant UCP maps. We establish this theory in Stinespring and Paschke dilations of completely positive maps. We end this note by mentioning some examples of UCP maps admitting a decomposition into $G$-invariant UCP maps. 
\end{abstract}

\maketitle

\section{Introduction}
Consider a group $G$ acting on a C*-algebra $A$. A $G$-invariant state on a C*-algebra $A$ is a state which remains invariant under the group action. A detailed study of the $G$-invariant and $G$-ergodic decompositions of $G$-invariant states on a C*-algebra may be found in \cite[Section 3]{BR1}, where the set of extreme points, that is, the set of $G$-ergodic states has been characterized. The set of extreme points of $G$-invariant states forms a subset of pure states of a C*-algebra. In this article, we attempt to establish a decomposition theory for unital completely positive (UCP) maps that are invariant under a group of *-automorphisms of a C*-algebra $A$. However, in this setting, the theory differs while describing the set of extreme points of $G$-invariant UCP maps from the classical setting (as seen in \cite[Section 3]{BR1}) because of the non-commutativity of the codomain. To counter this difference, we use the uniqueness of the minimal Stinespring dilation of UCP maps followed by Arveson's operator theoretic techniques to identify the set of extreme points of $G$-invariant UCP maps. En route, we show that a certain invariance property of the generalized orthogonal measures ensures that the measure is supported on the subset of $G$-invariant UCP maps. Following this, the maximality of such measures give an integral decomposition of $G$-invariant UCP maps with the generalized orthogonal measures supported on the set of ergodic UCP maps. This generalizes the classical theory of decomposition of $G$-invariant states on a C*-algebra to the setting of UCP maps from a C*-algebra $A$ into the non-commutative codomain $B(H)$. 

Throughout this article, we consider $A$ as a unital, separable C*-algebra, with $H$ denoting a separable Hilbert space. The algebra of all bounded operators on the Hilbert space $H$ is denoted by $B(H)$. Let $G$ be a group, and denote by $\text{Aut}(A)$ the group of *-automorphisms of the C*-algebra $A$. We consider a group action of $G$ on $A$ by a group representation of $G$ on $\text{Aut}(A)$, and we denote this by
\begin{equation*}
\tau : G \rightarrow \text{Aut}(A).
\end{equation*}
A UCP map $\phi : A \rightarrow B(H)$ that satisfies the condition $\phi(a) = \phi(\tau_g(a))$ for all $g \in G$ and for all $a \in A$ is called a $G$-invariant UCP map. We fix the following notations for this article: 
\begin{align*}
\ucpa &:= \{ \phi : A \rightarrow B(H) ~|~ \phi ~ \text{is a UCP map} \} \\
\ucpag &:= \{ \phi \in \ucpa ~|~ \phi ~ \text{is $G$-invariant} \}
\end{align*}
The set $\ucpag$ is a compact (in BW topology) and convex subset of the locally convex topological vector space of all completely bounded maps from $A$ into $B(H)$. Let $\text{ext}(\ucpag)$ denote the set of all extreme points of $\ucpag$. We call an element of the set $\text{ext}(\ucpag)$ as a $G$-ergodic UCP map.

We establish the theory of $G$-invariant decomposition of UCP maps by using the theory of generalized orthogonal measures on the space of UCP maps developed in \cite{BK3}. A $G$-invariant decomposition of UCP maps is exhibited by a particular subclass of generalized orthogonal measures. This subclass can be identified in a one-to-one correspondence with a particular class of abelian subalgebras of the commutant of the minimal Stinespring dilation of the UCP map. This one-to-one correspondence uses the uniqueness of the minimal Stinespring dilation of completely positive maps. Further, to complete the picture of $G$-ergodic decomposition of UCP maps, one needs to characterize the set $\text{ext}(\ucpag)$. This has been done by introducing Radon--Nikodym type results in the space of $G$-invariant UCP maps. We explicitly characterize the set $\text{ext}(\ucpag)$, by using Arvesons's operator theoretic techniques in Stinespring and Paschke dilations of completely positive maps. This accomplishes the barycentric decomposition in the space of $G$-invariant UCP maps. One may refer to \cite{BK2} for a similar type of techniques of characterizing extreme points of a compact convex set.  
  
We organize this article as follows: in Section \ref{sec;p}, we recall some notations, definitions, and results that are required for later sections. Firstly, we recall results from Stinespring dilation and Paschke dilation of completely positive maps followed by a brief introduction to the theory of generalized orthogonal measures from \cite{BK3}. In Section \ref{sec;giged}, we examine the cases for $G$-invariant and $G$-ergodic decomposition of $G$-invariant UCP maps. For this, we make the use of the theory of generalized orthogonal measures developed in \cite{BK3}. In Section \ref{sec;geucpm}, first we establish the Radon--Nikodym type results in the space of $G$-invariant completely positive maps in Stinespring as well as Paschke dilation of completely positive maps. Using these results, we characterize the set $\text{ext}(\ucpag)$. The basic idea of the proof is inspired by Arveson's \cite{Arveson1} operator theoretic proof of describing the extreme points of completely positive maps having ranges in $B(H)$. We end the article with Section \ref{sec;egoi} by providing some examples of UCP maps admitting a decomposition into $G$-invariant UCP maps.

\section{Preliminaries}\label{sec;p}

This section is divided in two subsections. In the first subsection, we recall the results from Stinespring and Paschke dilation of completely positive maps and Radon--Nikodym type results in both the dilations. One may find the detailed study on this topic in \cite{Arveson1, Paschke, Paulsen}.

\subsection{Dilations of Completely Positive Maps:}
To begin with we recall the Stinespring dilation theorem.

\begin{theorem}\cite[Theorem 4.1]{Paulsen}\label{thm;sd}
Let $A$ be a unital C*-algebra and $H$ be a Hilbert space. Let $\phi : A \rightarrow B(H)$ be a completely positive map. Then there exists a Hilbert space $K$, a bounded linear map $V : H \rightarrow K$, and a unital *-homomorphism $\rho : A \rightarrow B(K)$ such that
\begin{equation*}
\phi(a) = V^*\rho(a)V \hspace{10pt} \text{for all $a \in A$}. 
\end{equation*} 
Moreover, the set $\{ \rho(a)Vh : a \in A, h \in H \}$ spans a dense subspace of $K$, and if $\phi$ is unital, then $V : H \rightarrow K$ is an isometry.
\end{theorem}

The triple $(\rho, V, K)$ from Theorem \ref{thm;sd} is called as the minimal Stinespring representation for $\phi$. The following proposition states that for a given completely positive map $\phi$, the minimal Stinespring representation is unique upto an unitary equivalence.  

\begin{proposition}\cite[Proposition 4.2]{Paulsen} \label{prop;umsd}
Let $A$ be a unital C*-algebra and $H$ be a Hilbert space. Let $\phi : A \rightarrow B(H)$ be a completely positive map. For $i =1, 2$, let $(\rho_i, V_i, K_i)$ be two minimal Stinespring representations for $\phi$. Then there exits a unitary $U : K_1 \rightarrow K_2$ such that $UV_1 = V_2$, and $U \rho_1 U^* = \rho_2$. 
\end{proposition}

Let $A$ be a unital separable C*-algebra and $H$ be a separable Hilbert space. Consider the set
\begin{equation*}
\cpa := \{ \phi : A \rightarrow B(H) ~|~ \phi ~~ \text{is completely positive} \}.
\end{equation*} 
Define a partial order on $\cpa$ by $\phi_1 \leq \phi_2$, if $\phi_2 - \phi_1 \in \cpa$. For a fixed $\phi \in \cpa$, consider the set
\begin{equation*}
[0, \phi] = \{ \theta \in \cpa ~~ | ~~ \theta \leq \phi \}.
\end{equation*} 
Let $V^* \rho V$ be the minimal Stinespring dilation of $\phi$, where $V$ be a bounded map from $H$ to $K$, and $\rho$ be representation of $A$ on $K$. For a fixed $T \in \racom$, define $\phi_T := V^* \rho T V$. Then $\phi_T$ is a map from $A$ into $B(H)$. Arveson proved the following Radon--Nikodym type theorem for completely positive maps in \cite{Arveson1}:

\begin{theorem} \cite[Theorem 1.4.2]{Arveson1} \label{thm;arnd}
There is an affine order isomorphism of the partially ordered convex set of operators $\{T \in \rho(A)^\prime ~|~ 0 \leq T \leq  1_K \}$ onto $[0, \phi]$, which is given by the map 
\begin{equation*}
T \mapsto \phi_T = V^* \rho  T V.
\end{equation*}
\end{theorem}

We also use the Paschke dilation of completely positive maps on Hilbert modules. We quickly recall a few results about the Paschke dilation of completely positive maps that are required in this article. 

\begin{definition} \cite[Definition 2.1]{Paschke} \label{def;phm}
Let $B$ be a unital C*-algebra. Then a pre-Hilbert $B$-module $X$ is a right $B$-module equipped with a conjugate-bilinear map $\la \cdot, \cdot \ra : X \times X \rightarrow B$ satisfying:
\begin{enumerate}
\item $\la x, x \ra \geq 0$ for all $x \in X$;
\item $\la x, x \ra = 0$ if and only if $x = 0$;
\item $\la x, y \ra = \la y, x \ra^*$ for all $x, y \in X$;
\item $\la x.b, y \ra = \la x, y \ra b$ for all $x, y \in X$ and $b \in B$
\end{enumerate}	
\end{definition}

For a pre-Hilbert $B$-module $X$, we define a norm on $X$ by $\| x \|_X = \|\la x, x \ra \|^{\frac{1}{2}}.$ 

\begin{definition} \cite[Definition 2.4]{Paschke} \label{def;hm}
A pre-Hilbert $B$-module $X$ which is complete with respect to $\| \cdot \|_X$ is called as a Hilbert $B$-module. 
\end{definition}

Let $X$ be a Hilbert $B$-module. Then we consider the following set of operators on $X$.

\begin{definition} \cite{Paschke} \label{def;ao}
A bounded operator $T$ on $X$ is called as an \textit{adjointable operator}, if there exists a bounded operator $T^*$ on $X$ such that $\la Tx, y \ra = \la x, T^*y \ra$ for all $x, y \in X$
\end{definition}

Let $\mathcal{P}(X)$ denote the set of all \textit{adjointable operators} on $X$. Then the set $\mathcal{P}(X)$ forms a C*-algebra \cite{Paschke}. Suppose $A$ is a unital C*-algebra and $\sigma: A \rightarrow \mathcal{P}(X)$ is a *-representation of $A$ on $X$. Now, for $e \in X$, define a map $\phi : A \rightarrow B$ by
\begin{equation*}
\phi(a) := \langle \sigma(a)e, e \rangle.
\end{equation*}
Then $\phi$ is a completely positive map. Paschke proved in \cite{Paschke} that every completely positive map from $A$ into $B$ arises this way. 

\begin{theorem} \cite[Theorem 5.2]{Paschke} \label{thm;pd}
Let $A$ and $B$ be unital C*-algebras, and $\phi \colon A \rightarrow B$ be a completely positive map. Then there exists a Hilbert $B$-module $X$, a *-representation $\sigma \colon A \rightarrow \mathcal{P}(X)$, and $e \in X$ such that $\phi(a) = \langle \sigma(a)e, e \rangle$ for all $a \in A$. The set $\{ \sigma(a)(eb) ~|~ a \in A, b \in B \}$ spans a dense subspace of $X$.
\end{theorem}

This dilation is called as Paschke dilation of completely positive map. The property that the set $\{ \sigma(a)(eb) ~|~ a \in A, ~ b \in B \}$ spans a dense subspace of $X$ makes this dilation unique in the following sense. Since we were not able to locate a proof of the uniqueness in the literature, here we include the proof of the uniqueness for the sake of continuation. 

\begin{theorem}\label{thm;umpd}
Let $A$ and $B$ be unital C*-algebras, and $\phi \colon A \rightarrow B$ be a completely positive map. Suppose $(\sigma_i, e_i, X_i)$ for $i = 1 , 2$ be two triples in Paschke dilation for $\phi$, as given in Theorem \ref{thm;pd}. Then there exists a Hilbert module isomorphism $W : X_1 \rightarrow X_2$ such that $We_1 = e_2$, and $W \sigma_1(a) W^* = \sigma_2(a)$ for all $a \in A$.
\end{theorem}
\begin{proof}
We claim that the required Hilbert module isomorphism $W : X_1 \rightarrow X_2$ is given by:
\begin{equation*}
W\left (\sum_{i=1}^{n}\sigma_1(a_i)(e_1b_i) \right) = \sum_{i=1}^{n}\sigma_2(a_i)(e_2b_i).
\end{equation*}
We will verify that the above formula yields a well-defined isometry from $X_1$ to $X_2$. 
	
To this end, note that
\begin{align*}
\left \| \sum_{i=1}^{n} \sigma_1(a_i)(e_1b_i) \right \|^2 &= \left \|  \sum_{i=1}^{n} \sum_{j=1}^{n}  \la \sigma_1(a_i)(e_1b_i),  \sigma_1(a_j)(e_1b_j)  \ra \right  \| \\
&=  \left \| \sum_{i=1}^{n} \sum_{j=1}^{n}  \la \sigma_1(a^*_ja_i)e_1,  e_1  \ra b_ib^*_j  \right \| \\
&=  \left \|  \sum_{i=1}^{n} \sum_{j=1}^{n} \phi(a^*_ja_i)b_ib^*_j \right \|.
\end{align*} 
Similarly, we get
\begin{equation*}
\left \| \sum_{i=1}^{n}\sigma_2(a_i)(e_2b_i) \right \|^2 = \left  \|  \sum_{i=1}^{n} \sum_{j=1}^{n} \phi(a^*_ja_i)b_ib^*_j\right \|.
\end{equation*}
	
Since by the minimality condition, $W$ will have a dense range and hence onto.
Thus, $W$ is an isometry and this completes the proof.
\end{proof}

For a unital C*-algebra $B$, and a pre-Hilbert $B$-module $X$ define the set $X^\prime$ as follows:
\begin{equation*}
X^\prime := \{ \tau: X \rightarrow B ~|~ \tau ~ \text{is a bounded $B$-module map} \}.
\end{equation*} 
For each $x \in X$, we define a unique map $\hat{x} \colon X \rightarrow B$ as $\hat{x}(y) := \langle y, x \rangle$. If $X^\prime = \hat{X} = \{\hat{x} ~|~ x\in X \}$, then $X$ is said to be self-dual. If $B$ is a von-Neumann algebra, then $X^\prime$ is also a Hilbert $B$-module. Moreover, the $B$-valued inner product on $X$ extends to a $B$-valued inner product on $X^\prime$.

\begin{theorem} \cite[Theorem 3.2]{Paschke} \label{thm;xprime}
Let $B$ be a von-Neumann algebra and $X$ be a pre-Hilbert $B$-module. Then the $B$-valued inner product $\langle \cdot, \cdot \rangle$ on $X \times X$ extends to $X^\prime \times X^\prime$ in such a way as to make $X^\prime$ into a self-dual Hilbert $B$-module. Also, it satisfies $\langle \hat{x}, \tau \rangle = \tau(x)$ for $x \in X$, $\tau \in X^\prime$.
\end{theorem}

We have $X$ as a submodule of $X^\prime$. Then the following proposition extends an adjointable operator on $X$ to an adjointable operator on $X^\prime$:

\begin{proposition} \cite[Proposition 3.6, Corollary 3.7]{Paschke} \label{prop;pext}
Let $B$ be a von-Neumann algebra and $X$ be a pre-Hilbert $B$-module, then each $T \in \mathcal{P}(X)$ extends to a unique $\tilde{T} \in \mathcal{P}(X^\prime)$. The map $T \mapsto \tilde{T}$ is a *-isomorphism of $\mathcal{P}(X)$ into $\mathcal{P}(X^\prime)$. Moreover, $\tilde{T}\hat{x} = \widehat{Tx}$ for all $x \in X$.
\end{proposition}
The last equality $\tilde{T}\hat{x} = \widehat{Tx}$, for all $x \in X$ can be observed from the proof of \cite[Proposition 3.6]{Paschke}.

Let $A$ be a unital C*-algebra and $B$ be a von Neumann algebra. Suppose $\phi : A \rightarrow B$ is a completely positive map. By Theorem \ref{thm;pd}, we get the triple $(X, \sigma, e)$ as the Paschke dilation of $\phi$. Then using Proposition \ref{prop;pext}, we can define a *-representation $\tilde{\sigma}: A \rightarrow \mathcal{P}(X^\prime)$ by $\tilde{\sigma}(a) := \widetilde{\sigma(a)}$. For $T \in \mathcal{P}(X^\prime)$, define a linear map $\phi_T : A \rightarrow B$ by $\phi_T(a) := \langle T \tilde{\sigma}(a) \hat{e}, \hat{e} \rangle$ for $a \in A$. In the setting of Hilbert module, we have the following Radon--Nikodym type result due to W. Paschke.

\begin{theorem} \cite[Proposition 5.4]{Paschke} \label{thm;prnd}
There is an affine order isomorphism of $\{ T \in  \tilde{\sigma}(A)^\prime \hspace{5pt}| \hspace{5pt} 0\leq T \leq 1_{X^\prime} \}$ onto $[0, \phi] = \{ \theta \in CP(A, B) \hspace{5pt} | \hspace{5pt} \theta \leq \phi \}$ which is given by the map $T \mapsto \phi_T$. 
\end{theorem}

\subsection{Generalized Orthogonal Measures:}
Here we quickly recall results from the theory of orthogonal decomposition of UCP maps established in \cite{BK3}. The article \cite{BK3} extensively uses results from the theory of disintegration of representations of separable C*-algebra. For a detail introduction to the theory of disintegration of representations of separable C*-algebra, one may refer to \cite{DixC, DixvNA, KR2, Takesaki1}.

Let $A$ be a unital, separable C*-algebra, and $H$ be a separable Hilbert space. We set the following notation:
\begin{equation*}
\ucpa := \{ \phi : A \rightarrow B(H) ~|~ \phi ~~ \text{is unital and completely positive} \}.
\end{equation*}  
The set $\ucpa$ is a compact, convex subset of all completely bounded maps from $A$ into $B(H)$. Suppose $M_1(\ucpa)$ denotes the set of all positive Borel measures on $\ucpa$ with norm 1. For a fixed $\phi \in \ucpa$, we define the set
\begin{equation*}
M_\phi(\ucpa) := \left \{ \mu \in M_1(\ucpa) ~\mid~ \phi = \int_{\ucpa} \phi^\prime \, \mathrm{d} \mu \right \},
\end{equation*}
where the integral is taken in a weak sense. That is, for all $a \in A$ and $h_1, h_2 \in H$, we have $\la \phi(a) h_1, h_2 \ra = \int_{\ucpa} \la \phi^\prime(a)h_1, h_2 \ra  \, \mathrm{d} \mu$. If this happens, then we say that the $\phi$ is the barycenter of $\mu$. Now we recall the notion of orthogonality between two completely positive maps from $A$ having ranges into $B(H)$. 

\begin{definition}\cite[Definition 3.1]{BK3}\label{def;ocp}
Let $\phi_1$ and $\phi_2$ be two completely positive maps from $A$ into $B(H)$. Suppose $\phi_1 + \phi_2 = \phi$ and  $V^* \rho V$ is the minimal Stinespring dilation of $\phi$. Then we say, $\phi_1$ is orthogonal to $\phi_2$, and denoted by $\phi_1 \perp \phi_2$, if there exists a projection $P \in \racom$ such that 
\begin{equation*}
\phi_1 =  V^* P \rho V  \hspace{10pt} \text{and} \hspace{10pt} \phi_2 =  V^* (1-P) \rho V.
\end{equation*}
\end{definition}

Using the above definition, for fixed $\phi \in \ucpa$ we recall the definition of a generalized orthogonal measure.

\begin{definition}\cite[Definition 3.2]{BK3}\label{def;om}
Let $\mu \in M_\phi(\ucpa)$. Then we say that the measure $\mu$ is a generalized orthogonal measure if, for every Borel measurable subset $E$ of $\ucpa$, we get 
\begin{equation*}
\int\limits_{E} \phi^\prime \, \mathrm{d} \mu \perp \int\limits_{\ucpa \setminus E} \phi^\prime \, \mathrm{d} \mu.
\end{equation*}
\end{definition}

Hereafter in this article, a generalized orthogonal measure will be referred to as an orthogonal measure (on $\ucpa$) for ease of readability. We denote the set of all orthogonal measures with the barycenter $\phi$ by $\mathcal{O}_\phi(\ucpa)$. We have the following containment between the three sets:
\begin{equation*}
\mathcal{O}_\phi(\ucpa) \subseteq M_\phi(\ucpa) \subseteq M_1(\ucpa).
\end{equation*}

Let $\phi$ be in $\ucpa$, and $V^* \rho V$ be the minimal Stinespring dilation of $\phi$. For an arbitrary but fixed $\mu \in M_\phi(\ucpa)$, we give a unique map from $\text{L}^\infty(\ucpa, \mu)$ into $\racom$ satisfying certain conditions. On $\text{L}^\infty(\ucpa, \mu)$ we consider $\sigma(\text{L}^\infty, \text{L}^1)$ topology, that is $f_\alpha \rightarrow f$ in $\sigma(\text{L}^\infty, \text{L}^1)$ topology if and only if
\begin{equation*}
\int\limits_{\ucpa} f_\alpha g \, \mathrm{d} \mu \rightarrow \int\limits_{\ucpa} f g \, \mathrm{d} \mu
\end{equation*}
for all $g \in \text{L}^1(\ucpa, \mu)$. The algebra $\racom$ is considered with weak operator topology (WOT).

\begin{lemma}\cite[Lemma 3.4]{BK3}\label{lem;kmu}
Let $\phi \in \ucpa$ and $\mu$ is an element of $M_\phi(\ucpa)$. Suppose $V^* \rho V$ is the minimal Stinespring dilation of $\phi$. Then there exists a unique map 
\begin{equation*}
k_\mu : \text{L}^\infty(\ucpa, \mu) \rightarrow \racom
\end{equation*}
given by 
\begin{equation*}
\la k_\mu(f) \rho(a) Vh_1, \rho(b) Vh_2 \ra = \int\limits_{\ucpa} f(\phi^\prime) \la \phi^\prime (b^*a)h_1, h_2 \ra \, \mathrm{d} \mu,
\end{equation*}
where $a, b \in A$ and $h_1, h_2 \in H$. The map is positive, contractive and continuous in $\sigma(\text{L}^\infty, \text{L}^1) - \text{WOT}$. 
\end{lemma}

For fixed $\phi \in \ucpa$, the following proposition characterizes the orthogonal measures $\mu$ belonging to $M_\phi(\ucpa)$ by using the map $k_\mu$ defined above.

\begin{proposition}\cite[Proposition 3.5]{BK3}\label{prop;om}
Let $\mu \in M_\phi(\ucpa)$ and $V^* \rho V$ be the minimal Stinespring dilation of $\phi$. Then the following statements are equivalent
\begin{enumerate}
\item The measure $\mu$ is in $\mathcal{O}_\phi(\ucpa)$.
\item The map $f \mapsto k_\mu(f)$ is a *-isomorphism from $\text{L}^\infty(\ucpa, \mu)$ onto the range of $k_\mu$ in $\acom$.
\item The map $f \mapsto k_\mu(f)$ is a *-homomorphism.
\end{enumerate}
\end{proposition}

\begin{remark}\cite[Remark 3.6]{BK3}\label{rem;avna}
If one of the conditions (and hence the other two) in the previous proposition is satisfied, then $\clB_\mu = \{ k_\mu(f) ~|~ f \in \text{L}^\infty(\ucpa, \mu) \}$ is an abelian von Neumann subalgebra of $\racom$.
\end{remark}

Now we consider a special class of abelian subalgebras called orthogonal abelian subalgebras defined in \cite{BK3}.
\begin{definition}\cite[Definition 3.10]{BK3}\label{def;oasa}
Let $\phi \in \ucpa$, where $V^* \rho V$ be the minimal Stinespring dilation of $\phi$ with $V : H \rightarrow K$. Let $\clB$ be an abelian subalgebra of $\acom$. Suppose $K$ and $\rho$ disintegrate as $\diX K_p \, \mathrm{d} \nu$ and $\diX \rho_p \, \mathrm{d} \nu$ respectively with respect to the abelian subalgebra $\clB$. Then we say $\clB$ is an orthogonal abelian subalgebra of $\acom$ if;
\begin{enumerate}
\item the operator $V : H \rightarrow K = \diX K_p \, \mathrm{d} \nu$ can be written as $V = \diX V_p \, \mathrm{d} \nu$, where $V_p : H \rightarrow K_p$ is an isometry for almost every $p \in X$;
\item the abelian algebra $\text{L}^\infty(\ucpa, \mu_\clB)$ is isomorphic to an algebra $\text{L}^\infty(X, \nu)$, where $\mu_\clB$ is the pushforward measure defined on $\ucpa$ using the measurable map say $g : X \setminus X_0 \rightarrow \ucpa$ defined as $p \mapsto V_p^* \rho_p V_p$, where $X_0$ is the $\nu$ measure zero set consisting of those $p$ such that $V_p$ is not isometry or $\rho_p$ is not unital.
\end{enumerate} 
\end{definition}

Lastly, we mention the main theorem in \cite{BK3}, that gives a one-to-one correspondence between orthogonal measures with orthogonal abelian subalgebras.
\begin{theorem}\cite[Theorem 3.12]{BK3}\label{thm;omoas}
Let $\phi \in \ucpa$ and $V^* \rho V$ be the minimal Stinespring dilation of $\phi$. Then there is a one-to-one correspondence between the following sets:
\begin{enumerate}
\item the orthogonal measures $\mu$ with $\phi$ as the barycenter, that is $\mu \in  \mathcal{O}_\phi(\ucpa)$.
\item the orthogonal abelian subalgebras, $\clB \subseteq \acom$.
\end{enumerate} 
\end{theorem}

\section{A $G$-invariant and $G$-ergodic Decomposition of UCP Maps} \label{sec;giged}
Let $G$ be a group and $A$ be a separable, unital C*-algebra. Let $\tau : G \rightarrow \text{Aut}(A)$ be an action of the group $G$ on the C*-algebra $A$. Then for a fixed separable Hilbert space $H$, we consider the following set of all $G$-invariant UCP maps from $A$ with ranges in $B(H)$:
\begin{equation*}
\ucpag := \left \{ \phi \in \ucpa ~|~ \phi(a) = \phi(\tau_g(a)), \forall g \in G, a \in A \right \}.
\end{equation*}
Now we consider an example of a $G$-invariant UCP map. 
\begin{example} \label{ex;Ginv}
Let $G$ be a (compact) group and let $\lambda$ be the Haar measure on $G$. Consider $H$ as a separable Hilbert space, and $A$ as a C*-subalgebra of $B(H)$. Suppose $\tau : G \rightarrow \text{Aut}(A)$ is a continuous action of the group $G$ on $A$ given by 
\begin{equation*}
\tau_g(a) := U_g^* a U_g,
\end{equation*}
where $\{ U_g \}_{g \in G}$ is a family of unitaries in $A$. Consider a unital completely positive map $\phi_G : A \rightarrow B(H)$ given by 
\begin{equation*}
\phi_G(a) := \int_{G} U_g^*aU_g \, \mathrm{d} \lambda,
\end{equation*}
where the integral is observed in a weak sense. That is, for $h_1, h_2 \in H$ we have
\begin{equation*}
\la \phi_G(a)h_1, h_2 \ra =  \int_{G} \la U_g^*aU_gh_1, h_2 \ra \, \mathrm{d} \lambda.
\end{equation*}
For an arbitrary but fixed $g_0 \in G$, we have
\begin{eqnarray*}
\la \phi_G(\tau_{g_0}(a))h_1, h_2 \ra &=&  \int_{G} \la U_g^*\tau_{g_0}(a)U_gh_1, h_2 \ra \, \mathrm{d} \lambda \\
&=& \int_{G} \la U_g^*U_{g_0}^* a U_{g_0}U_gh_1, h_2 \ra \, \mathrm{d} \lambda \\
&=& \int_{G} \la U_{gg_0}^* a U_{g_0g}h_1, h_2 \ra \, \mathrm{d} \lambda \\
&=& \la \phi_G(a)h_1, h_2 \ra 
\end{eqnarray*} 
This shows that the map $\phi_G$ is a $G$-invariant UCP map.
\end{example}

Moreover, the set $\ucpag$ is a compact (in BW topology) and convex subset of the locally convex topological vector space of all completely bounded maps from $A$ to $B(H)$. Let $\text{ext}(\ucpag)$ denote the set of all extreme points of $\ucpag$.

\begin{definition} \label{def;e}
An extreme point of the set $\ucpag$, that is an element of the set $\text{ext}(\ucpag)$, is called as a $G$-ergodic UCP map. 
\end{definition}

Let $\phi \in \ucpag$, and the minimal Stinespring dilation of $\phi$ be given by $V^* \rho V$, where $V: H \rightarrow K$ be a bounded linear operator. Since $\phi \in \ucpag$, for $a\in A$ and $g \in G$, we have 
\begin{equation*}
\phi(a) = V^* \rho(a) V = \phi(\tau_g(a)) = V^* \rho(\tau_g(a)) V.
\end{equation*} 
This gives, for fixed $g \in G$ the minimal Stinespring dilation of $\phi$ can also be given by $V^* \rho \circ \tau_g V$. Now, by the uniqueness of the minimal Stinespring dilation (see Proposition \ref{prop;umsd}), we get a unitary operator, say $U_g : K \rightarrow K$ satisfying 
\begin{equation*}
U_gV = V  ~~~ \text{and} ~~~ U_g \rho U^*_g = \rho \circ \tau_g.
\end{equation*}
This gives a representation of a group $G$ on the Hilbert space $K$, denoted by $U_\phi : G \rightarrow B(K)$, and is defined as:
\begin{equation*}
U_\phi(g) := U_g.
\end{equation*}
Moreover, we have
\begin{equation*}
U_\phi(g)V = V  ~~~ \text{and} ~~~ U_\phi(g) \rho(a) U_\phi(g)^* = \rho \circ \tau_g(a), ~~~ \forall a \in A, ~ g \in G.
\end{equation*}

For a fixed $\phi \in \ucpag$, the following theorem characterizes the orthogonal measures on $\ucpa$ with barycenter $\phi$, for which the corresponding orthogonal abelian algebra reside in the commutant of $\rho(A)$ and $U_\phi(G)$. This characterization leads to a $G$-invariant decomposition of $\phi$. 

\begin{theorem} \label{thm;inv}
Let $\tau : G \rightarrow \text{Aut}(A)$ be an action of the group $G$ on the C*-algebra $A$. Let $\phi$ be an element of $\ucpag$. Then there is a one-to-one correspondence between the following:
\begin{enumerate}
\item the orthogonal measures $\mu \in \mathcal{O}_\phi(\ucpa)$ which satisfy 
\begin{equation*}
\int\limits_{\ucpa} f(\phi^\prime) \la \phi^\prime(\tau_{g^{-1}}(ab))h_1, h_2 \ra \, \mathrm{d} \mu = \int\limits_{\ucpa} f(\phi^\prime) \la \phi^\prime(ab)h_1, h_2 \ra \, \mathrm{d} \mu;
\end{equation*}
for all $a, b \in A$, $h_1, h_2 \in H$, $g \in G$, and $f \in \text{L}^\infty(\ucpa, \mu)$.
\item the orthogonal abelian subalgebras of $\{\rho(A) \cup U_\phi(G)\}^\prime$.
\end{enumerate}
\end{theorem}
\begin{proof}
Let $\mu$ be an orthogonal measure on $\ucpa$ with the barycenter $\phi$. Then by Theorem \ref{thm;omoas}, we get the orthogonal abelian subalgebra corresponding to $\mu$, say $\mathcal{B}_\mu$ of $\racom$. By applying Lemma \ref{lem;kmu}, and Proposition \ref{prop;om}, we get a *-isomorphism $k_\mu : \text{L}^\infty(\ucpa, \mu) \rightarrow \racom$ satisfying 
\begin{equation*}
\la k_\mu(f) \rho(b) Vh_1, \rho(a) Vh_2 \ra = \int\limits_{\ucpa} f(\phi^\prime) \la \phi^\prime(a^*b)h_1, h_2 \ra \, \mathrm{d} \mu
\end{equation*}
for all $a, b \in A$, $h_1, h_2 \in H$, and $f \in \text{L}^\infty(\ucpa, \mu)$.
Moreover, from Remark \ref{rem;avna}, we have 
\begin{equation*}
\mathcal{B}_\mu = \{ k_\mu(f) ~|~ f \in \text{L}^\infty(\ucpa, \mu)  \}.
\end{equation*}

Note that the one-to-one correspondence between orthogonal measures with the barycenter $\phi$ and orthogonal abelian subalgebras of $\racom$ is already established in Theorem \ref{thm;omoas}. Therefore, to prove the theorem, it is sufficient to show that the orthogonal measure satisfies the invariance property given above if and only if the corresponding orthogonal abelian subalgebra is contained in $\{\rho(A) \cup U_\phi(G)\}^\prime$. 

First, we assume that the orthogonal measure $\mu$ satisfies 
\begin{equation*}
\int\limits_{\ucpa} f(\phi^\prime) \la \phi^\prime(\tau_{g^{-1}}(ab))h_1, h_2 \ra \, \mathrm{d} \mu = \int\limits_{\ucpa} f(\phi^\prime) \la \phi^\prime(ab)h_1, h_2 \ra \, \mathrm{d} \mu
\end{equation*}
for all $a, b \in A$, $h_1, h_2 \in H$, $g \in G$, and $f \in \text{L}^\infty(\ucpa, \mu)$. Then, we prove that the corresponding orthogonal abelian subalgebra $\mathcal{B}_\mu$ commutes with $U_\phi(g)$ for all $g \in G$. For this, consider
\begin{flalign*}
\la \rho(a) k_\mu(f) \rho(b) Vh_1, Vh_2 \ra &= \la k_\mu(f) \rho(ab) Vh_1, Vh_2 \ra \\
&= \int\limits_{\ucpa} f(\phi^\prime) \la \phi^\prime(ab)h_1, h_2 \ra \, \mathrm{d} \mu \\
&= \int\limits_{\ucpa} f(\phi^\prime) \la \phi^\prime(\tau_{g^{-1}}(ab))h_1, h_2 \ra \, \mathrm{d} \mu \\
&= \la k_\mu(f) \rho(\tau_{g^{-1}}(ab)) Vh_1, Vh_2 \ra \\
&= \la k_\mu(f) \rho(\tau_{g^{-1}}(a) \tau_{g^{-1}}(b)) Vh_1, Vh_2 \ra \\
&= \la \rho(\tau_{g^{-1}}(a)) k_\mu(f) \rho(\tau_{g^{-1}}(b)) Vh_1, Vh_2 \ra \\
&= \la U_\phi(g^{-1}) \rho(a)U_\phi(g) k_\mu(f) U_\phi(g^{-1}) \rho(b) U_\phi(g) Vh_1, Vh_2 \ra  \\
&= \la \rho(a)U_\phi(g) k_\mu(f) U_\phi(g^{-1}) \rho(b) Vh_1, Vh_2 \ra.
\end{flalign*}
By the minimality of the Hilbert space  $K$, we have
\begin{equation*}
k_\mu(f) = U_\phi(g) k_\mu(f) U_\phi(g^{-1})
\end{equation*}
Since this is true for all $g \in G$, we get $k_\mu(f)  \in U_\phi(G)^\prime$, that is $\mathcal{B}_\mu \subseteq U_\phi(G)^\prime$. Therefore, we have $\mathcal{B}_\mu \subseteq \{\rho(A) \cup U_\phi(G)\}^\prime$.

Conversely, suppose there exists an orthogonal abelian subalgebra, say $\mathcal{B}$ of $\{\rho(A) \cup U_\phi(G)\}^\prime$, and $\mu_{\mathcal{B}}$ be the orthogonal measure corresponding  to $\mathcal{B}$. Then we show that the measure $\mu_{\mathcal{B}}$ satisfies the following invariance condition:
\begin{equation*}
\int\limits_{\ucpa} f(\phi^\prime) \la \phi^\prime(\tau_{g^{-1}}(ab))h_1, h_2 \ra  \mathrm{d} \mu_{\mathcal{B}} = \int\limits_{\ucpa} f(\phi^\prime) \la \phi^\prime(ab)h_1, h_2 \ra \mathrm{d} \mu_{\mathcal{B}}.
\end{equation*}
For verifying this, consider
\begin{align*}
\la \rho(a) k_\mu(f) \rho(b) Vh_1, Vh_2 \ra &= \la k_\mu(f) \rho(ab) Vh_1, Vh_2 \ra \\
&= \int\limits_{\ucpa} f(\phi^\prime) \la \phi^\prime(ab)h_1, h_2 \ra \, \mathrm{d} \mu_{\mathcal{B}}.
\end{align*}
Also, we have
\begin{align*}
\la \rho(a) k_\mu(f) \rho(b) Vh_1, Vh_2 \ra
&= \la \rho(a) U_\phi(g) k_\mu(f) U_\phi(g^{-1}) \rho(b) Vh_1, Vh_2 \ra \\
&= \la U_\phi(g^{-1}) \rho(a) U_\phi(g) k_\mu(f) U_\phi(g^{-1}) \rho(b) U_\phi(g) Vh_1, Vh_2 \ra  \\
&= \la \rho(\tau_{g^{-1}}(a)) k_\mu(f) \rho(\tau_{g^{-1}}(b)) Vh_1, Vh_2 \ra \\
&= \la k_\mu(f) \rho(\tau_{g^{-1}}(ab)) Vh_1, Vh_2 \ra \\
&= \int\limits_{\ucpa} f(\phi^\prime) \la \phi^\prime(\tau_{g^{-1}}(ab))h_1, h_2 \ra \, \mathrm{d} \mu_{\mathcal{B}}
\end{align*}
Therefore, we have 
\begin{equation*}
\int\limits_{\ucpa} f(\phi^\prime) \la \phi^\prime(\tau_{g^{-1}}(ab))h_1, h_2 \ra \, \mathrm{d} \mu_{\mathcal{B}} = \int\limits_{\ucpa} f(\phi^\prime) \la \phi^\prime(ab)h_1, h_2 \ra \, \mathrm{d} \mu_{\mathcal{B}}.
\end{equation*}
This complete the proof.
\end{proof}

An orthogonal measure satisfying the invariance condition given in Theorem \ref{thm;inv} leads to a $G$-invariant decomposition of the UCP map. This has been proved in the following proposition.

\begin{proposition} \label{prop;supp}
Let $\tau : G \rightarrow \text{Aut}(A)$ be an action of a group $G$ on a C*-algebra $A$. Let $\phi$ be an element of $\ucpag$. Suppose $\mu \in \mathcal{O}_\phi(\ucpa)$. Then the measure $\mu$ satisfies 
\begin{equation*}
\int\limits_{\ucpa} f(\phi^\prime) \la \phi^\prime(\tau_{g^{-1}}(ab))h_1, h_2 \ra \, \mathrm{d} \mu = \int\limits_{\ucpa} f(\phi^\prime) \la \phi^\prime(ab)h_1, h_2 \ra \, \mathrm{d} \mu;
\end{equation*}
for all $a, b \in A$, $h_1, h_2 \in H$, $g \in G$, and $f \in \text{L}^\infty(\ucpa, \mu)$ if and only if the support of $\mu$ is contained in $\ucpag$. Moreover, if the measure $\mu$ is maximal in $M_\phi(\ucpa)$, then $\mu$ is supported on $\text{ext}(\ucpag)$.
\end{proposition}
\begin{proof}
Let $\phi \in \ucpag$, and $\mu$ be an element of $\mathcal{O}_\phi(\ucpa)$. Suppose the support of the measure $\mu$ is contained in $\ucpag$, then for $g \in G$, and $a, b \in A$ we have 
\begin{equation*}
\phi^\prime(\tau_{g^{-1}}(ab)) = \phi^\prime(ab) ~~ \text{for almost every} ~~ \phi^\prime.
\end{equation*}
Therefore, for all $a, b \in A$, $h_1, h_2 \in H$, $g \in G$ and $f \in \text{L}^\infty(\ucpa, \mu)$, we get
\begin{equation*}
\int\limits_{\ucpa} f(\phi^\prime) \la \phi^\prime(\tau_{g^{-1}}(ab))h_1, h_2 \ra \, \mathrm{d} \mu = \int\limits_{\ucpa} f(\phi^\prime) \la \phi^\prime(ab)h_1, h_2 \ra \, \mathrm{d} \mu.
\end{equation*}

Conversely, suppose for an arbitrary but fixed $a, b \in A$, $h_1, h_2 \in H$, and $g \in G$ we have,
\begin{equation*}
\int\limits_{\ucpa} f(\phi^\prime) \la \phi^\prime(\tau_{g^{-1}}(ab))h_1, h_2 \ra \, \mathrm{d} \mu = \int\limits_{\ucpa} f(\phi^\prime) \la \phi^\prime(ab)h_1, h_2 \ra \, \mathrm{d} \mu
\end{equation*}
is true for all $f \in \rm{L}^\infty(\ucpa, \mu)$. Then, we necessarily have 
\begin{equation*}
\la \phi^\prime(\tau_{g^{-1}}(ab))h_1, h_2 \ra  = \la \phi^\prime(ab)h_1, h_2 \ra ~~ \text{for almost every} ~~ \phi^\prime.
\end{equation*}
This implies $\phi^\prime \in \ucpag$ for almost every $\phi^\prime$. Therefore, the support of $\mu$ is contained in $\ucpag$.

The last part of the proposition follows because every maximal measure is supported on the set of extreme points (see \cite[Theorem 4.1.7]{BR1}). If the measure $\mu$ satisfies equivalent conditions given in the proposition, and if $\mu$ is maximal in $M_\phi(\ucpa)$, then $\mu$ is supported (as the C*-algebra $A$ and the Hilbert space $H$ are separable) on $\text{ext}(\ucpag)$. 
\end{proof}

\section{$G$-ergodic UCP Maps}\label{sec;geucpm}
In the previous section, we analysed a $G$-invariant and a $G$-ergodic decomposition of a $G$-invariant UCP map. In the case of a $G$-ergodic decomposition, the measure is supported on the set of extreme points of $\ucpag$. Therefore, the theory of a $G$-ergodic decomposition is incomplete unless we characterize the set $\text{ext}(\ucpag)$. The characterization of extreme points of the set $\ucpag$ is given in Theorem \ref{thm;egiucp}. But before stating the theorem, we prove the following lemma that gives a Radon--Nikodym type result in the space of $G$-invariant UCP maps (in Stinespring dilation). We use the following notations to state the lemma.

\begin{align*}
\cpa &:= \{ \phi : A \rightarrow B(H) ~|~ \phi ~~ \text{is completely positive} \} \\
\cpag &:= \left \{ \phi \in \cpa ~|~ \phi(a) = \phi(\tau_g(a)), \forall g \in G, a \in A \right \}
\end{align*}

\begin{lemma}\label{lem;rnd}
Let $\phi \in \ucpag$, and $V^* \rho V$ be the minimal Stinespring dilation of $\phi$. Then, the map $T \mapsto \phi_T$ is an affine isomorphism of the partially ordered convex set of operators $\{ T \in  (\rho(A) \cup U_\phi (G) )^\prime ~|~ 0 \leq T \leq 1_K\}$ onto $[0, \phi] \cap \cpag$.
\end{lemma}
\begin{proof}
By using Theorem \ref{thm;arnd}, we know that the map $T \mapsto \phi_T$ is an
affine isomorphism of the partially ordered convex set of operators $\{ T \in  \rho(A)^\prime ~|~ 0 \leq T \leq 1_K\}$ onto $[0, \phi]$. Therefore, to prove the lemma it is sufficient to show $\phi_T \in [0, \phi] \cap \cpag$ if and only if $T \in  (\rho(A) \cup U_\phi (G) )^\prime$ with $0 \leq T \leq 1_K$.

First, assume $\phi_T \in [0, \phi] \cap \cpag$, then $\phi_T \leq \phi$ and $\phi_T(a) = \phi_T(\tau_g(a))$ for all $a \in A$ and $g \in G$. Since $\phi_T \leq \phi$, we have from Theorem \ref{thm;arnd}, that $T \in \rho(A)^\prime$ and $0 \leq T \leq 1_K$. Then using the assumption $\phi_T \in \cpag$, we show that $T \in U_\phi(G)^\prime$.  For this, consider for $a, b \in A$, $h_1, h_2 \in H$ and $g \in G$ we have,
\begin{align*}
\la T \rho(a) Vh_1, \rho(b) Vh_2 \ra &= \la V^* T \rho(b^*a) Vh_1, h_2 \ra  \\
&= \la V^* T U_\phi(g) \rho(b^*a) U_\phi(g)^* Vh_1, h_2 \ra \\
&= \la \rho(a) U_\phi(g)^* Vh_1, \rho(b) U_\phi(g)^* T Vh_2 \ra \\
&= \la \rho(a) U_\phi(g)^* Vh_1, U_\phi(g)^* \rho(\tau_g(b)) T Vh_2 \ra \\
&= \la \rho(a) U_\phi(g)^* Vh_1, U_\phi(g)^* T \rho(\tau_g(b)) Vh_2 \ra \\
&= \la T U_\phi(g) \rho(a) U_\phi(g)^* Vh_1, \rho(\tau_g(b)) Vh_2 \ra \\
&= \la T U_\phi(g) \rho(a) U_\phi(g)^* Vh_1, U_\phi(g) \rho(b) U_\phi(g)^* Vh_2  \ra \\
&= \la U_\phi(g)^* T U_\phi(g) \rho(a) U_\phi(g)^* Vh_1, \rho(b) Vh_2 \ra \\
&= \la U_\phi(g)^* T U_\phi(g) \rho(a) Vh_1, \rho(b) Vh_2 \ra
\end{align*}
By the minimality of the Hilbert space $K$, we get $T = U_\phi(g)^* T U_\phi(g)$. Since $g \in G$ was arbitrary, we get $T \in U_\phi(G)^\prime$. This proves, if $\phi_T \in [0, \phi] \cap \cpag$, then $T \in  (\rho(A) \cup U_\phi (G) )^\prime$ with $0 \leq T \leq 1_K$.

Next, assume $T \in  (\rho(A) \cup U_\phi (G) )^\prime$ and $0 \leq T \leq 1_K$. Then we show that $\phi_T \in [0, \phi] \cap \cpag$. Since $T \in \racom$ and $0 \leq T \leq 1_K$, we have $\phi_T \in [0, \phi]$. Using the assumption $T \in U_\phi (G)^\prime$, we show that $\phi_T \in \cpag$. For this, consider an arbitrary but fixed $a \in A$ and $g \in G$, then we have
\begin{align*}
\la \phi_T(a) h_1, h_2 \ra &= \la V^* T \rho(a) V h_1, h_2 \ra \\
&= \la T \rho(a) Vh_1, Vh_2 \ra \\
&= \la T \rho(a) U_\phi(g) Vh_1, U_\phi(g)^* Vh_2 \ra \\
&= \la U_\phi(g) T \rho(a) U_\phi(g)^* Vh_1, h_2 \ra \\
&= \la T U_\phi(g) \rho(a) U_\phi(g)^* Vh_1, h_2 \ra \\
&= \la V^* T \rho(\tau_g(a)) Vh_1, h_2 \ra \\
&= \la \phi_T(\tau_g(a)) h_1, h_2 \ra.
\end{align*}
This shows that $\phi_T \in \cpag$. Therefore, if $T \in  (\rho(A) \cup U_\phi (G) )^\prime$ with $0 \leq T \leq 1_K$, then $\phi_T \in [0, \phi] \cap \cpag$. 
\end{proof}

Now we prove a similar result in the case of Paschke dilation of completely positive maps. First, we set some preliminary notations to state the result. Consider $\phi \in \ucpag$, then by applying Theorem \ref{thm;pd} we get the triple $(X, \sigma, e)$, where $X$ is a Hilbert-$B(H)$ module, $e \in X$, and $\sigma : A \rightarrow \mathcal{P}(X)$ be a *-representation, such that 
\begin{equation*}
\phi(a) = \langle \sigma(a)e, e \rangle ~~ \forall ~~ a \in A.
\end{equation*}
Since $\phi \in \ucpag$, for $a \in A$ and $g \in G$, we have 
\begin{equation*}
\phi(a) = \langle \sigma(a)e, e \rangle = \langle \sigma(\tau_g(a))e, e \rangle = \phi(\tau_g(a)).
\end{equation*}
This gives, for fixed $g \in G$ the Paschke dilation of $\phi$ can also be given by the triple $(X, \sigma \circ \tau_g, e)$.  Now, by the uniqueness given in Theorem \ref{thm;umpd}, we get a Hilbert module isomorphism, say $W_g : X \rightarrow X$ satisfying $W_ge = e$ and $W_g \sigma W^*_g = \sigma \circ \tau_g$. Since this is true for all $g \in G$, we get a representation of the group $G$ on the Hilbert module $X$ denoted by $W_\phi : G \rightarrow \mathcal{P}(X)$, and which is defined as:
\begin{equation*}
W_\phi(g) := W_g.
\end{equation*}
Moreover, this satisfies
\begin{equation*}
W_\phi(g)e = e \hspace{10pt} \text{and} \hspace{10pt} W_\phi(g) \sigma(a) W_\phi(g)^* = \sigma \circ \tau_g(a), ~~~ \forall a \in A, ~ g \in G.
\end{equation*}  

Consider the set 
\begin{equation*}
X^\prime := \{ \tau: X \rightarrow B(H) ~|~ \tau ~ \text{is a bounded $B$-module map} \}.
\end{equation*} 
Then Theorem \ref{thm;xprime} extends the $B(H)$-valued inner product $\langle \cdot, \cdot \rangle$ to $X^\prime \times X^\prime$ in such a way as to make $X^\prime$ into a self-dual Hilbert $B(H)$-module. Further, using Proposition \ref{prop;pext} we define a representation of the group $G$ on $X^\prime$ as $\tilde{W_\phi}: G \rightarrow \mathcal{P}(X^\prime)$, where $\tilde{W_\phi}(g)$ is defined by $\widetilde{W_\phi(g)}$. Also, corresponding to $\sigma$, we have a *-representation $\tilde{\sigma}: A \rightarrow \mathcal{P}(X^\prime)$ defined by $\tilde{\sigma}(a) := \widetilde{\sigma(a)}$. Moreover, we get
\begin{equation*}
\tilde{W_\phi}(g) \tilde{\sigma}(a) \tilde{W_\phi}(g)^* = \widetilde{\sigma}\circ \tau_g(a) \hspace{10pt} \forall a \in A, ~ g \in G.
\end{equation*}
and 
\begin{equation*}
\tilde{W_\phi}(g)\hat{e} = \hat{e} \hspace{10pt} \forall ~ g \in G.
\end{equation*}
From Section \ref{sec;p}, recall that for $T \in \mathcal{P}(X^\prime)$, one can define a linear map $\phi_T : A \rightarrow B(H)$ by $\phi_T(a) := \langle T \tilde{\sigma}(a) \hat{e}, \hat{e} \rangle$ for $a \in A$. We are now ready to state and prove the following Radon--Nikodym type result for $\phi \in \ucpag$ in the setting of Paschke dilation of completely positive maps.

\begin{lemma}\label{lem;prnd}
Let $\phi \in \ucpag$ and $(X, \sigma, e)$ be the triple in the Paschke dilation of $\phi$. Then, the map $T \mapsto \phi_T$ is an affine order isomorphism of $\{ T \in  (\tilde{\sigma}(A) \cup \tilde{W_\phi}(G))^\prime ~|~ 0 \leq T \leq 1_{X^\prime} \}$ onto $[0, \phi] \cap \cpag$.
\end{lemma}
\begin{proof}
From Theorem \ref{thm;prnd}, we know that the map $T \mapsto \phi_T$ is an affine isomorphism of $\{ T \in \tilde{\sigma}(A)^\prime ~|~ 0\leq T \leq 1_{X^\prime} \}$ onto $[0, \phi]$. Therefore, to prove the lemma it is enough to show $\phi_T \in [0, \phi] \cap \cpag$ if and only if $T \in  (\tilde{\sigma}(A) \cup \tilde{W_\phi}(G))^\prime$ with $0 \leq T \leq 1_{X^\prime}$. 

Let $\phi_T \in [0, \phi] \cap \cpag$. Then $\phi_T \leq \phi$ and $\phi_T(a) = \phi_T(\tau_g(a))$ for all $a \in A$ and $g \in G$. Since $\phi_T \leq \phi$, from Theorem \ref{thm;prnd}, we get $T \in \tilde{\sigma}(A)^\prime$ and $0 \leq T \leq 1_{X^\prime}$. Now we use the assumption $\phi_T \in \cpag$ to show that $T \in \tilde{W_\phi}(G)^\prime$. For $a, b \in A$ and $g \in G$ we have, 
\begin{align*}
\la T \tilde{\sigma}(a) \hat{e}, \tilde{\sigma}(b) \hat{e} \ra &= \la T \tilde{\sigma}(b^*a) \hat{e}, \hat{e} \ra  \\
&= \la T \tilde{W_\phi}(g) \tilde{\sigma}(b^*a) \tilde{W_\phi}(g)^* \hat{e}, \hat{e} \ra \\
&= \la \tilde{\sigma}(a) \tilde{W_\phi}(g)^* \hat{e}, \tilde{\sigma}(b) \tilde{W_\phi}(g)^* T \hat{e} \ra \\
&= \la \tilde{\sigma}(a) \tilde{W_\phi}(g)^*\hat{e}, \tilde{W_\phi}(g)^* \tilde{\sigma}(\tau_g(b)) T \hat{e} \ra \\
&= \la \tilde{\sigma}(a) \tilde{W_\phi}(g)^*\hat{e}, \tilde{W_\phi}(g)^* T \tilde{\sigma}(\tau_g(b))\hat{e} \ra \\
&= \la T \tilde{W_\phi}(g) \tilde{\sigma}(a) \tilde{W_\phi}(g)^* \hat{e}, \tilde{\sigma}(\tau_g(b)) \hat{e} \ra \\
&= \la T \tilde{W_\phi}(g) \tilde{\sigma}(a) \tilde{W_\phi}(g)^* \hat{e}, \tilde{W_\phi}(g) \tilde{\sigma}(b) \tilde{W_\phi}(g)^* \hat{e}  \ra \\
&= \la \tilde{W_\phi}(g)^* T \tilde{W_\phi}(g) \tilde{\sigma}(a) \tilde{W_\phi}(g)^* \hat{e}, \tilde{\sigma}(b) \hat{e} \ra \\
&= \la \tilde{W_\phi}(g)^* T \tilde{W_\phi}(g) \tilde{\sigma}(a) \hat{e}, \tilde{\sigma}(b) \hat{e} \ra.
\end{align*}
By the minimality of the Hilbert module $X$, we get 
\begin{equation*}
T = \tilde{W_\phi}(g)^* T \tilde{W_\phi}(g),
\end{equation*} 
that is $T \in \tilde{W_\phi}(G)^\prime$. Therefore, if $\phi_T \in [0, \phi] \cap \cpag$, then $T \in  (\tilde{\sigma}(A) \cup \tilde{W_\phi}(G))^\prime$ with $0 \leq T \leq 1_{X^\prime}$. 

For the converse, assume $T \in  (\tilde{\sigma}(A) \cup \tilde{W_\phi}(G))^\prime$ and $0 \leq T \leq 1_{X^\prime}$, then we show that $\phi_T \in [0, \phi] \cap \cpag$. Since $T \in \tilde{\sigma}(A) ^\prime$, and $0 \leq T \leq 1_{X^\prime}$, we have $\phi_T \in [0, \phi]$. Now we use the assumption $T \in \tilde{W_\phi} (G)^\prime$ to show that $\phi_T \in \cpag$. For this, consider an arbitrary but fixed $a \in A$ and $g \in G$, then we have
\begin{align*}
\phi_T(a) &= \la T \tilde{\sigma}(a) \hat{e}, \hat{e} \ra \\
&= \la T \tilde{\sigma}(a) \tilde{W_\phi}(g) \hat{e}, \tilde{W_\phi}(g)^* \hat{e} \ra \\
&= \la \tilde{W_\phi}(g) T \tilde{\sigma}(a) \tilde{W_\phi}(g)^* \hat{e}, \hat{e} \ra \\
&= \la T \tilde{W_\phi}(g) \tilde{\sigma}(a) \tilde{W_\phi}(g)^* \hat{e}, \hat{e} \ra \\
&= \la T \tilde{\sigma}(\tau_g(a)) \hat{e}, \hat{e} \ra \\
&= \phi_T(\tau_g(a))
\end{align*}
This shows that $\phi_T \in \cpag$. Therefore, if $T \in  (\tilde{\sigma}(A) \cup \tilde{W_\phi}(G))^\prime$ with $0 \leq T \leq 1_{X^\prime}$, then $\phi_T \in [0, \phi] \cap \cpag$.  
\end{proof}

Now we are ready to prove the main theorem of this section which characterizes extreme points of the set $\ucpag$. But before that, we recall a definition due to Arveson in \cite{Arveson1}.

\begin{definition}
Let $H$ be a Hilbert space and $H_0$ be a closed subspace of $H$. Let $P$ denote the projection of $H$ onto $H_0$, then, we say that $H_0$ is a faithful subspace of $H$ for a von Neumann algebra $M$, if, $PT_{|H_0} = 0$ for $T \in M$, then $T = 0$. 
\end{definition}

\begin{theorem}\label{thm;egiucp}
Let $\tau : G \rightarrow \text{Aut}(A)$, and $\phi \in \ucpag$. Let $V^* \rho V$ be the minimal Stinespring dilation, and $(X, \sigma, e)$ be the Paschke dilation of $\phi$. Then the following statements are equivalent:
\begin{enumerate}
\item the UCP map $\phi \in \text{ext}(\ucpag)$;
\item if $\phi_0 \in [0, \phi] \cap t\ucpag$, then $\phi_0 = t\phi$, where $0 < t < 1$;
\item the subspace $[VH]$ is a faithful subspace of the Hilbert space $K$ for the von Neumann algebra $\{ \rho(A) \cup U_\phi (G) \}^\prime$; 
\item if any $T \in (\tilde{\sigma}(A) \cup \tilde{W_\phi}(G))^\prime$ satisfies $\langle T \hat{e}, \hat{e} \rangle = 0$, then $T=0$. 
\end{enumerate}
\end{theorem}
\begin{proof}
We start by proving (1) $\iff$ (2). Suppose $\phi$ is an extreme point of $\ucpag$ and $\phi_0 \in [0, \phi] \cap t\ucpag$, where $0 < t < 1$. Then $\phi_0 = t \phi_1 \leq \phi$ for some $\phi_1 \in \ucpag$. Using Theorem \ref{thm;arnd}, we get $\phi_0 = t\phi_1 = V^* T \rho V$ for some $T \in \racom$ with $0 \leq T \leq  1_K$. Since $\phi_0 \in t \ucpag$, the operator $T \in U_\phi(G)^\prime$ that follows from Lemma \ref{lem;rnd}. 

Then consider the operator $1_K - T$. Since $T \in \{ \rho(A) \cup U_\phi (G) \}^\prime$ and $0 \leq T \leq  1_K$ we get $(1_K - T) \in \{ \rho(A) \cup U_\phi (G) \}^\prime$, and also, $0 \leq (1_K - T) \leq  1_K$. Then, using Lemma \ref{lem;rnd} we get $V^* (1_K - T) \rho V \in [0, \phi] \cap \cpag$. Define $\phi_2 := V^* \frac{(1_K - T)}{1 - t} \rho V$, then  
\begin{align*}
\phi_2(1_A) &= V^*  \rho(1_A) \frac{(1_K - T)}{1 - t} V  \\
&= \frac{1}{1 - t} (V^* 1_K   V - V^* T V) \\
&= \frac{(1-t)1_H}{1 - t} = 1_H
\end{align*} 
This implies $\phi_2 \in \ucpag$. 

For $\phi_1, \phi_2 \in \ucpag$, we get $\phi = t\phi_1 + (1-t)\phi_2$. Since $\phi$ is an extreme point, $\phi = \phi_1 = \phi_2$. Hence, $\phi_0 = t \phi_1 = t \phi$.

Conversely, if $\phi_0 \in [0, \phi] \cap t\ucpag$, then $\phi_0 = t\phi$, where $0 < t < 1$. Suppose $\phi = \lambda\phi_1 + (1 - \lambda) \phi_2$, for $\phi_1, \phi_2 \in \ucpag$, and $0 < \lambda < 1$. Then $\lambda\phi_1 \in [0, \phi] \cap \lambda\ucpag$, hence from the assumption, we get $\lambda\phi_1 = \lambda\phi$, and similarly $(1 - \lambda) \phi_2 = (1- \lambda) \phi$. This implies $\phi_1 = \phi = \phi_2$.

Now we prove (1) $\iff$ (3):
We first assume that the subspace $[VH]$ is a faithful subspace of $K$ for the algebra $\{ \rho(A) \cup U_\phi (G) \}^\prime$. Suppose $\phi_1, \phi_2 \in \ucpag$, and $0 < \lambda < 1$ such that $\phi = \lambda \phi_1 + (1- \lambda) \phi_2$. Since $\lambda \phi_1 \leq \phi$, using Theorem \ref{thm;arnd} we get, $\lambda \phi_1 = V^* T \rho V$ for some $T \in \racom$ such that $0 \leq T \leq 1_K$. Moreover, Lemma \ref{lem;rnd} implies that $T \in U_\phi(G)^\prime$. Let $P$ be the projection of the Hilbert space $K$ onto $[VH]$. Let $a_1, a_2 \in A$ and $h_1, h_2, \in H$, 
\begin{align*}
\left \langle P T V h_1, V h_2 \right \rangle &= \left \langle P T \rho(1_A) V h_1, V h_2 \right \rangle \\
&= \left \langle T \rho(1_A)) V h_1, V h_2 \right \rangle \\
&= \left \langle V^* T \rho(1_A) V h_1,  h_2 \right \rangle \\
&= \left \langle \lambda h_1,  h_2 \right \rangle \\
&= \lambda \left \langle V^* V h_1,  h_2 \right \rangle \\
&=  \left \langle \lambda V h_1, V h_2 \right \rangle
\end{align*}
Therefore, $PT = \lambda 1_{[VH]}$. This implies that $P(T - \lambda 1_K)_{[VH]}$ is $0$. Since $[VH]$ is a faithful subspace of $K$ for the algebra $\{ \rho(A) \cup U_\phi (G) \}^\prime$, and $T - \lambda 1_K \in \{ \rho(A) \cup U_\phi (G) \}^\prime$, we get, $T = \lambda 1_K$. Hence, $\lambda \phi_1 = V^* T \rho V = \lambda \phi$, and $\phi_1 = \phi = \phi_2$.

Conversely, let us assume that $\phi$ is an extreme point of $\ucpag$. Define a positive linear map $\mu: \{ \rho(A) \cup U_\phi (G) \}^\prime \rightarrow B([VH])$ by, $\mu (T) := PT_{[VH]}$. We need to show that $\mu$ is injective. That is, if for some $T \in \{ \rho(A) \cup U_\phi (G) \}^\prime$ we have $\mu(T)= 0$, then $T = 0$. But as $\mu$ preserves adjoints we can consider $T = T^*$. So we may choose positive scalars $s, t$ such that $\frac{1}{4} 1_K \leq sT + t 1_K \leq \frac{3}{4} 1_K$. Then, as $\mu$ is positive, $\frac{1}{4} \mu(1_K) \leq \mu(sT + t 1_K) \leq \frac{3}{4} \mu(1_K)$. So we get: $\frac{1}{4} 1_{[VH]} \leq  t 1_{[VH]}\leq \frac{3}{4} 1_{[VH]}$. This implies that $0 < t < 1$. Next, define: 
\begin{equation*}
\phi_1 := V^*(sT + t 1_K) \rho V ~ \rm{and} ~ \phi_2 := V^*(1_K - (sT + t 1_K)) \rho V.
\end{equation*} 
But we know that $T \in \{ \rho(A) \cup U_\phi (G) \}^\prime$, and therefore $sT + t 1_K$ and $1_K - (sT + t 1_K)$ both operators belong in $\{ \rho(A) \cup U_\phi (G) \}^\prime$. Then applying Lemma \ref{lem;rnd} with observing a fact that $0 < sT + t 1_K <1_K$ and $0 < 1_K - (sT + t 1_K) <1_K$ we get $\phi_1, \phi_2 \in [0, \phi] \cap \cpag$. Moreover, we have $\phi_1 + \phi_2 = \phi$. Since $\mu(T)= 0$, we get
\begin{equation*}
\phi_1(1_A) = V^*(sT + t 1_K) \rho(1_A) V = sV^* T V + tV^*V = t1_H 
\end{equation*}
and similarly $\phi_2(1_A) = (1-t) 1_H$. This implies $\frac{1}{t} \phi_1, \frac{1}{1 - t} \phi_2 \in \ucpag$. 

Now, recall that $\phi = \phi_1 + \phi_2$. Writing $\phi = t (\frac{1}{t} \phi_1) + (1 - t) (\frac{1}{1 - t} \phi_2)$ and using the hypothesis that $\phi$ is an extreme point, we get: 
\begin{equation*}
\phi = \frac{1}{t} \phi_1 = \frac{1}{1 - t} \phi_2.
\end{equation*}
This gives us
\begin{equation*}
\phi_1 = V^*(sT + t 1_K) \rho V = V^* t1_K \rho V = t\phi \leq \phi.
\end{equation*}
By the order preserving affine isomorphism from Theorem \ref{thm;arnd}, and Lemma \ref{lem;rnd}, we must have $sT + t 1_K = t 1_K$, which in turn implies that $T=0$. This shows that, the map $\mu$ is injective, and therefore $[VH]$ is a faithful subspace of $K$ for $\{ \rho(A) \cup U_\phi (G) \}^\prime$.

Finally, we prove $(1)\Leftrightarrow (4)$. Let $(X, \sigma, e)$ be the Paschke dilation such that $\phi(\cdot) = \langle \sigma(\cdot)e, e \rangle$ be as in Theorem \ref{thm;pd}. From Theorem \ref{thm;prnd}, we have $\phi(\cdot) = \langle \sigma(\cdot)e, e \rangle = \langle 1_{X^\prime} \tilde{\sigma}(\cdot)\hat{e}, \hat{e} \rangle$. Suppose for  $T \in (\tilde{\sigma}(A) \cup \tilde{W_\phi}(G))^\prime$ we have $\langle T \hat{e}, \hat{e} \rangle = 0$, then $T = 0$. Then we show that $\phi$ is an extreme point of $\ucpag$. If $\phi = \lambda \phi_1 + (1 - \lambda) \phi_2$ for $\phi_1, \phi_2 \in \ucpag$ and $0 < \lambda < 1$, then $\lambda \phi_1 \leq \phi$. Therefore, by Lemma \ref{lem;prnd}, we get $T_1 \in (\tilde{\sigma}(A) \cup \tilde{W_\phi}(G))^\prime$ with $0 \leq T_1 \leq 1_{X^\prime}$ such that $\lambda \phi_1(a) = \langle T_1 \tilde{\sigma}(a)\hat{e}, \hat{e} \rangle$ for all $a \in A$. Hence, we have:
\begin{equation*}
\lambda \phi_1(1_A) = \langle T_1 \tilde{\sigma}(1_A)\hat{e}, \hat{e} \rangle = \lambda 1_H = \lambda \langle 1_{X^\prime} \tilde{\sigma}(1_A)\hat{e}, \hat{e} \rangle.
\end{equation*}
This implies $\langle (T_1 - \lambda 1_{X^\prime})\hat{e}, \hat{e} \rangle = 0$. Since $T_1 \in (\tilde{\sigma}(A) \cup \tilde{W_\phi}(G))^\prime$ we get $T_1 - \lambda 1_{X^\prime} \in (\tilde{\sigma}(A) \cup \tilde{W_\phi}(G))^\prime$. By hypothesis, we get  $T_1 = \lambda 1_{X^\prime}$. Therefore, $\lambda \phi_1(\cdot)= \langle \lambda \tilde{\sigma}(\cdot)\hat{e}, \hat{e} \rangle = \lambda \phi(\cdot)$. Hence, $\phi_1 = \phi = \phi_2$.

Conversely, assume that $\phi$ is an extreme point of $\ucpag$. Let $T \in (\tilde{\sigma}(A) \cup \tilde{W_\phi}(G))^\prime$ be such that $\langle T \hat{e}, \hat{e} \rangle = 0$. Now define $\mu : (\tilde{\sigma}(A) \cup \tilde{W_\phi}(G))^\prime \rightarrow B(H)$ by $S \mapsto \langle S \hat{e}, \hat{e} \rangle$. By assumption, we have $\mu(T) = 0$, then we prove that $T=0$. But as $\mu$ preserves adjoints we can consider $T = T^*$. Choose positive scalars $s, t$ such that $\frac{1}{4} 1_{X^\prime} \leq sT + t 1_{X^\prime} \leq \frac{3}{4} 1_{X^\prime}$. So, we get 
\begin{equation*}
\frac{1}{4} \mu(1_{X^\prime}) \leq \mu(sT + t 1_{X^\prime}) \leq \frac{3}{4} \mu(1_{X^\prime}).
\end{equation*} 
Since $\mu(T) = 0$, we have:
\begin{equation*}
\frac{1}{4} \langle 1_{X^\prime}\hat{e}, \hat{e} \rangle \leq  t \langle 1_{X^\prime}\hat{e}, \hat{e} \rangle  \leq \frac{3}{4} \langle 1_{X^\prime}\hat{e}, \hat{e} \rangle.
\end{equation*} 
This implies $0 < t < 1$. Define $\phi_i: A \rightarrow B(H)$ for $i=1, 2$ by 
\begin{equation*}
\phi_1(\cdot) := \langle (sT + t 1_{X^\prime}) \tilde{\sigma}(\cdot) \hat{e}, \hat{e} \rangle  \hspace{10pt} \text{and} \hspace{10pt}
\phi_2(\cdot) := \langle (1_{X^\prime} - (sT + t 1_{X^\prime})) \tilde{\sigma}(\cdot) \hat{e}, \hat{e} \rangle .
\end{equation*}
But, we know that $T \in (\tilde{\sigma}(A) \cup \tilde{W_\phi}(G))^\prime$. Therefore, $sT + t 1_{X^\prime}$ and $1_{X^\prime} - (sT + t 1_{X^\prime})$ both operators belong to $(\tilde{\sigma}(A) \cup \tilde{W_\phi}(G))^\prime$. Then applying Lemma \ref{lem;prnd} with observing a fact that $0 < sT + t 1_K <1_{X^\prime}$ and $0 < 1_{X^\prime} - (sT + t 1_{X^\prime}) <1_{X^\prime}$ we get $\phi_1, \phi_2 \in [0, \phi] \cap \cpag$. Moreover, we have $\phi_1 + \phi_2 = \phi$. Since $\mu(T)= 0$, we get
\begin{equation*}
\phi_1(1_A) = \langle (sT + t 1_{X^\prime}) \tilde{\sigma}(1_A) \hat{e}, \hat{e} \rangle  =  t1_H, 
\end{equation*}
and similarly $\phi_2(1_A) = (1-t) 1_H$. This implies $\frac{1}{t} \phi_1, \frac{1}{1 - t} \phi_2 \in \ucpag$. 

Now, recall that $\phi = \phi_1 + \phi_2$. Writing $\phi = t (\frac{1}{t} \phi_1) + (1 - t) (\frac{1}{1 - t} \phi_2)$ and using the hypothesis that $\phi$ is an extreme point, we get: 
\begin{equation*}
\phi = \frac{1}{t} \phi_1 = \frac{1}{1 - t} \phi_2.
\end{equation*}
This gives us
\begin{equation*}
\phi_1(\cdot) = \langle (sT + t 1_{X^\prime}) \tilde{\sigma}(\cdot) \hat{e}, \hat{e} \rangle = \langle  t 1_{X^\prime} \tilde{\sigma}(\cdot) \hat{e}, \hat{e} \rangle = t\phi(\cdot) \leq \phi.
\end{equation*}
By the order preserving affine isomorphism from Theorem \ref{thm;prnd}, and Lemma \ref{lem;prnd}, we must have $sT + t 1_{X^\prime} = t 1_{X^\prime}$, which in turn implies that $T=0$. 
\end{proof}

The previous theorem provides the necessary and sufficient conditions for the UCP map $\phi \in \ucpag$ to be an ergodic UCP map, that is, an element of the set $\text{ext}(\ucpag)$. We deduce that $\phi$ is an extreme point of $\ucpag$ if and only if the Hilbert space $[VH]$ is a faithful subspace for $\{ \rho(A) \cup U_{\phi} (G) \}^\prime$. Using this equivalent condition we derive the following corollary. We use the same notations as in the previous theorem.  

\begin{corollary}\label{cor;egiucp1}
Let $\tau : G \rightarrow \text{Aut}(A)$, and $\phi \in \ucpag$. Let $V^* \rho V$ be the minimal Stinespring dilation of $\phi$. If one of the following conditions is satisfied, then $\phi \in \text{ext}(\ucpag)$.  
\begin{enumerate}
\item The representation $\rho$ of the C*-algebra $A$ is irreducible.
\item The representation $U_\phi$ of the group $G$ is irreducible.
\end{enumerate}
\end{corollary}
\begin{proof}
Given that $V^* \rho V$ is the minimal Stinespring dilation of $\phi$, we have $[\rho(A)VH] = K$. Consequently, $[VH]$ is non-degenerate for $\{ \rho(A) \cup U_\phi (G) \}^{\prime\prime}$. Let $P$ be the projection of the Hilbert space $K$ onto $[VH]$.

First, assuming the representation $\rho$ of the C*-algebra $A$ is irreducible, the von Neumann algebra $\{ \rho(A) \cup U_\phi (G) \}^{\prime}$ consists of scalar multiples of identity operators on the Hilbert space $K$. This implies that the subspace $[VH]$ is invariant under $\{ \rho(A) \cup U_\phi (G) \}^\prime$, and hence $P \in \{ \rho(A) \cup U_\phi (G) \}^{\prime \prime}$. Let $T \in \{ \rho(A) \cup U_\phi (G) \}^\prime$. Assume $PT|_{[VH]} = 0$. Then for any $S \in \{ \rho(A) \cup U_\phi (G) \}^{\prime\prime}$, we have 
\begin{equation*}
0 = SPT|_{[VH]} = STP|_{[VH]} = ST|_{[VH]} = TS|_{[VH]}.
\end{equation*} 
Since $VH$ is non-degenerate for $\{ \rho(A) \cup U_\phi (G) \}^{\prime\prime}$, and $S$ is arbitrary element from $\{ \rho(A) \cup U_\phi (G) \}^{\prime\prime}$, we get $T = 0$. That is, the Hilbert space $[VH]$ is a faithful subspace for $\{ \rho(A) \cup U_{\phi} (G) \}^\prime$, and hence $\phi \in \text{ext}(\ucpag)$. 

Furthermore, if we assume the representation $U_\phi$ of the group $G$ is irreducible, then the von Neumann algebra $\{ \rho(A) \cup U_\phi (G) \}^{\prime}$ consists of scalar multiples of the identity operators on the Hilbert space $K$. This implies that the subspace $[VH]$ is invariant under $\{ \rho(A) \cup U_\phi (G) \}^\prime$, and hence $[VH]$ is a faithful subspace for $\{ \rho(A) \cup U_{\phi} (G) \}^\prime$. Consequently, $\phi \in \text{ext}(\ucpag)$. 
\end{proof}

We quickly recall that the notation $\cpa$ denotes the set of all completely positive maps from the C*-algebra $A$ to $B(H)$. A completely positive $\phi \in \cpa$ is called pure if, for any $\psi \in \cpa$ such that $\psi \leq \phi$, we have $\psi = t \phi$ for some $t \geq 0$ (\cite{Arveson1}).
 
\begin{corollary}\label{cor;egiucp2}
Let $\phi$ be an element of  $\ucpag$. Suppose $\phi$ is pure in the set of $\cpa$. Then $\phi$ belongs to $\text{ext}(\ucpag)$.
\end{corollary}
\begin{proof}
Let $V^* \rho V$ be the minimal Stinspring dilation of $\phi$. Since $\phi$ is pure in the set $\cpa$, we have the representation $\rho$ is irreducible \cite[Corollary 1.4.3]{Arveson1}. Then using Corollary \ref{cor;egiucp1} we get the result.
\end{proof}

We end this section with the following remark that discusses the link between the faithfulness of $[VH]$ for $\{ \rho(A) \cup U_{\phi} (G) \}^\prime$ with the non-degeneracy of $[VH]$ for $\{ \rho(A) \cup U_{\phi} (G) \}^{\prime \prime}$. W. Arveson has previously noted this relationship in a general context in \cite{Arveson1}. We include this observation here in a particular case for the sake of completeness.

\begin{remark}\cite[Chapter 1]{Arveson1}
Observe, if the subspace $[VH]$ is faithful for the von Neumann algebra $\{ \rho(A) \cup U_\phi (G) \}^\prime$, then $[VH]$ is separating for $\{ \rho(A) \cup U_\phi (G) \}^\prime$. However, the subspace $[VH]$ is separating for the von Neumann algebra $\{ \rho(A) \cup U_\phi (G) \}^\prime$ if and only if $[VH]$ is non-degenerate for the von Neumann algebra $\{ \rho(A) \cup U_{\phi} (G) \}^{\prime \prime}$. But conversely, if $[VH]$ is non-degenerate for $\{ \rho(A) \cup U_{\phi} (G) \}^{\prime \prime}$, then we do not necessarily have $[VH]$ to be faithful for $\{ \rho(A) \cup U_{\phi} (G) \}^{\prime}$. However, in the special case, when $[VH]$ is invariant under $\{ \rho(A) \cup U_{\phi} (G) \}^{\prime}$, we have $[VH]$ is faithful for $\{ \rho(A) \cup U_{\phi} (G) \}^\prime$ if and only if $[VH]$ is non-degenerate for $\{ \rho(A) \cup U_{\phi} (G) \}^{\prime \prime}$.
\end{remark}

\section{Examples of $G$-invariant Decomposition of UCP Maps} \label{sec;egoi}
In this section, we see some examples of a $G$-invariant decomposition of a map in $\ucpag$ for a finite dimensional Hilbert space $H$. First, we recall a few definitions and results from the classical case of the state space of a C*-algebra (see \cite[Section 4.3]{BR1} for more details). 

Let $A$ be a separable, unital C*-algebra and $S(A)$ be the state space of $A$. Let $G$ be a group and $\tau : G \rightarrow \text{Aut}(A)$ be an action of a group $G$ on $A$. Then $\omega \in S(A)$ is said to be $G$-invariant state if 
\begin{equation*}
\omega(a) = \omega(\tau_g(a))
\end{equation*}  
for all $a \in A$ and $g \in G$. The set of all $G$-invariant states is denoted by $S_G(A)$. This set is compact and convex subset of $S(A)$ in weak*-topology. The set denoted by $\text{ext}(S_G(A))$ consists of all the extreme points of $S_G(A)$, and an element of this set is called a $G$-ergodic state. 

Let $\omega \in S_G(A)$ and $(\pi_{\omega}, H_\omega, \Omega_\omega)$ be the corresponding GNS representation of $\omega$. Since $\omega \in S_G(A)$, for all $g \in G$, we have $(\pi_{\omega} \circ \tau_g, H_\omega, \Omega_\omega)$ be also the GNS representation of $\omega$. Then from the uniqueness of the GNS representation, for fixed $g \in G$, we get a unitary operator, say $U_{\omega}(g)$ on $H_\omega$ such that
\begin{equation*}
U_\omega(g) \pi_{\omega}(a) U_\omega(g)^* = \pi_{\omega}(\tau_g(a)) \hspace{10pt} \text{and} \hspace{10pt} U_\omega(g) \Omega_\omega = \Omega_\omega
\end{equation*} 
for all $a \in A$. Since this is true for all $g \in G$, we get a unitary representation of the group $G$ on $H_\omega$ which we denote by $U_\omega$. 

Corresponding to the action $\tau$, one may define an action $\tau^*$ of the group $G$ on the dual $A^*$. Using this, we define an action of the group $G$ on the algebra $C(S(A))$ of continuous functions on $S(A)$. Further, if $\mu$ is Baire measure on $S(A)$ which is invariant under this action, then one can extend this action to the algebra $\text{L}^\infty(S(A), \mu)$. We denote this action again by $\tau$ (a reader may refer \cite[Section 4.3.1]{BR1} for more details). We recall the following proposition from the classical setting.

\begin{proposition}\cite[Proposition 4.3.1]{BR1}
With the notations as above, we have a one-to-one correspondence between the following:
\begin{enumerate}
\item the orthogonal measures $\mu$, over $S(A)$, with the barycenter $\omega$, which satisfy the invariance condition $\mu(\tau_g(f_1)f_2) = \mu(f_1f_2)$ for all $f_1, f_2 \in \text{L}^\infty(S(A), \mu)$ and $g \in G$;	
\item the abelian von Neumann subalgebra $\mathcal{B}$ of $\{ \pi_{\omega}(A) \cup U_\omega(G) \}^\prime$.
\end{enumerate} 
\end{proposition}

Moreover, if one of the above conditions is true (and hence the other), then the support of $\mu$ is contained in the weak* closed subset $S_G(A)$. 

Now using the above result, we construct a few examples of maps in $\ucpag$ that satisfy the invariance condition given in Theorem \ref{thm;inv}. Then, Proposition \ref{prop;supp} will imply a $G$-invariant decomposition of such maps. The following construction of examples is similar to that of Examples given in \cite[Section 4]{BK3}. Here, we consider the similar set of examples with some additional constraints.

\begin{example} \label{ex;nn(E)}
Let $G$ be a group and $A$ be a unital separable C*-algebra. Let $\tau : G \rightarrow \text{Aut}(A)$ be an action of the group $G$ on $A$. Let $\omega \in S_G(A)$ with the GNS triple $(\pi_\omega, H_\omega, \Omega_\omega)$. For $n \in \bbN$, consider $\phi^n_\omega : M_n(A) \rightarrow M_n(\bbC)$ defined by 
\begin{equation*}
\phi^n_\omega ([a_{i,j}]) := [(\omega(a_{i,j}))_{i,j}].                                          
\end{equation*}
Then $\phi^n_\omega$ is a UCP map. Moreover, $\phi^n_\omega \in \ucpanng$ as $\omega$ is in $S_G(A)$. Let $V^* \rho V$ be the minimal Stinespring dilation of $\phi^n_\omega$, where $V : \bbC^n \rightarrow K$ be an isometry. 
	
For each $\omega^\prime \in S(A)$ we have a UCP map $\phi^n_{\omega^\prime} \in \ucpann$ which is defined similarly as above. Define a measurable function $\tilde{f} : S(A) \rightarrow \ucpann$ as:
\begin{equation*}
\tilde{f}(\omega^\prime) : = \phi^n_{\omega^\prime}.
\end{equation*}	
Let $\mu$ be a measure on $S(A)$ with the barycenter $\omega$. Then we have
\begin{equation*}
\phi^n_\omega ([a_{i,j}]) = [(\omega(a_{i,j}))_{i,j}] = \left [ \left (\int_{S(A)} \omega^\prime(a_{i,j}) \, \mathrm{d} \mu \right ) _{i,j}\right ].
\end{equation*}
Suppose $\tilde{\mu}$ is a pushforward measure on $\ucpann$ obtained by the function $\tilde{f}$. Then for $h, k \in \bbC^n$ we get
\begin{equation*}
\la \phi^n_\omega ([a_{i,j}]) h, k \ra = \int\limits_{\ucpann} \la \phi^\prime([a_{i,j}])h, k \ra \, \mathrm{d} \tilde{\mu}.
\end{equation*}
That is $\tilde{\mu} \in M_{\phi^n_\omega}(\ucpann)$. 	
	
Further, we assume that $\mu$ is an orthogonal measure on $S(A)$ with the barycenter $\omega$. Then there exists a *-isomorphism, given by $k_\mu : \text{L}^\infty (S(A), \mu) \rightarrow \pi_\omega(A)^\prime$ onto its range (see \cite[Proposition 4.1.22]{BR1}). Then from \cite[Example 4.1]{BK3}, we get $\tilde{\mu} \in \mathcal{O}_{\phi^n_\omega}(\ucpann)$. Equivalently, the map $k_{\tilde{\mu}} : \text{L}^\infty (\ucpann, \tilde{\mu}) \rightarrow \rho(M_n(A))^\prime$ is a *-isomorphism onto its range (see Proposition \ref{prop;om}). 
	
Next, we show that if $k_\mu(\text{L}^\infty (S(A), \mu))$ is contained in the algebra $\{ \pi_{\omega}(A) \cup U_\omega(G) \}^\prime$, then $k_{\tilde{\mu}}( \text{L}^\infty (\ucpann, \tilde{\mu}))$ is contained in $\{ \rho(M_n(A)) \cup U_\phi (G) \}^\prime$. To prove this, we have to show that $k_{\tilde{\mu}}(f) \in U_\phi(G)^\prime$ for all $f \in \text{L}^\infty (\ucpann, \tilde{\mu}))$. But from \cite[Example 4.1]{BK3} we know 
\begin{equation*}
k_{\tilde{\mu}}(f) = \begin{bmatrix}
k_\mu(f\circ \tilde{f})       &0 &\ldots  &0\\
0 & k_\mu(f\circ \tilde{f})      &\ddots  &\vdots\\
\vdots  &\ddots  &k_\mu(f\circ \tilde{f})       &0\\
0 &\ldots  &0 &k_\mu(f\circ \tilde{f})	
\end{bmatrix}_{n \times n}
\end{equation*}
and $\rho([a_{i,j}]) = [(\pi_\omega(a_{i,j}))_{i,j}]$.
Since, $\omega \in S_G(A)$ and $\phi^n_\omega \in \ucpanng$ for fixed $g \in G$, we have 
\begin{equation*}
U_\omega(g) \pi_{\omega}(a) U_\omega(g)^* = \pi_{\omega}(\tau_g(a)) \hspace{5pt} \text{and} \hspace{5pt}
U_\phi(g) \rho([a_{i,j}]) U_\phi(g)^* = \rho(\tau_g([a_{i,j}])).
\end{equation*}
From this, we identify $U_\phi(g)$ as
\begin{equation*}
\begin{bmatrix}
U_\omega(g)      &0 &\ldots  &0\\
0 & U_\omega(g)     &\ddots  &\vdots\\
\vdots  &\ddots  &U_\omega(g)       &0\\
0 &\ldots  &0 &U_\omega(g)
\end{bmatrix}_{n \times n}.
\end{equation*}
	
We have assumed that $k_\mu(\text{L}^\infty (S(A), \mu))$ is contained in $\{ \pi_{\omega}(A) \cup U_\omega(G) \}^\prime$, and in particular in  $U_\omega(G)^\prime$.  Then using the identification of $k_{\tilde{\mu}}(f)$ for $f \in \text{L}^\infty (\ucpann, \tilde{\mu}))$ and $U_\phi(g)$ from above, we get $k_{\tilde{\mu}}(f)$ belongs to $U_\phi (G)^\prime$ for all $f \in \text{L}^\infty (\ucpann, \tilde{\mu}))$. This shows that $k_{\tilde{\mu}}( \text{L}^\infty (\ucpann, \tilde{\mu}))$ is contained in $\{ \rho(M_n(A)) \cup U_\phi (G) \}^\prime$. Then by applying Theorem \ref{thm;inv} and Proposition \ref{prop;supp} we get a $G$-invariant decomposition of $\phi^n_\omega$.
\end{example}

The technique used in the following example is similar to that of the previous example. 

\begin{example} \label{ex;1n(E)}
We consider a group $G$ acting on a unital separable C*-algebra $A$, where the action is given by $\tau$. Let $\omega \in S_G(A)$ and $(\pi_\omega, H_\omega, \Omega_\omega)$ be the corresponding GNS triple . For $n \in \bbN$, define a map $\phi_{n,\omega} : A \rightarrow M_n(\bbC)$ by 
\begin{equation*}
\phi_{n,\omega} (a) := \begin{bmatrix}
\omega(a)      &0 &\ldots  &0\\
0 & \omega(a)    &\ddots  &\vdots\\
\vdots  &\ddots  &\omega(a)       &0\\
0 &\ldots  &0 &\omega(a)
\end{bmatrix}_{n \times n}.
\end{equation*}
Then $\phi_{n,\omega}$ is a UCP map. Moreover, $\phi_{n,\omega} \in \ucpang$. Let $V^* \rho V$ be the minimal Stinespring dilation of $\phi_{n,\omega}$, where $V : \bbC^n \rightarrow K$ be an isometry. Similarly, for each $\omega^\prime \in S(A)$, we have a UCP map $\phi_{n, \omega^\prime} \in \ucpan$ as defined above. Define a measurable function $\tilde{f} : S(A) \rightarrow \ucpan$ as:
\begin{equation*}
\tilde{f}(\omega^\prime) : = \phi_{n, \omega^\prime}.
\end{equation*}
Let $\mu$ be a measure on $S(A)$ such that $\omega = \int_{S(A)} \omega^\prime \, \mathrm{d} \mu$. Then we have
\begin{equation*}
\phi_{n, \omega} (a) = \begin{bmatrix}
\int\limits_{S(A)} \omega^\prime(a) \, \mathrm{d} \mu      &0 &\ldots  &0\\
0 & \int\limits_{S(A)} \omega^\prime(a) \, \mathrm{d} \mu   &\ddots  &\vdots\\
\vdots  &\ddots  &\int\limits_{S(A)} \omega^\prime(a) \, \mathrm{d} \mu      &0\\
0 &\ldots  &0 &\int\limits_{S(A)} \omega^\prime(a) \, \mathrm{d} \mu
\end{bmatrix}_{n \times n}.
\end{equation*}
Let $\tilde{\mu}$ be a pushforward measure on $\ucpan$ defined using the function $\tilde{f}$. Then for $h, k \in \bbC^n$ we get
\begin{equation*}
\la \phi_{n, \omega} (a) h, k \ra = \int\limits_{\ucpan} \la \phi^\prime(a)h, k \ra \, \mathrm{d} \tilde{\mu}.
\end{equation*}
That is $\tilde{\mu} \in M_{\phi_{n, \omega}}(\ucpan)$. 
	
Let $\mu$ be an orthogonal measure on $S(A)$ with the barycenter $\omega$. Then by \cite[Proposition 4.1.22]{BR1}, we get a *-isomorphism given by $k_\mu : \text{L}^\infty (S(A), \mu) \rightarrow \pi_\omega(A)^\prime$ onto its range. We know from \cite[Example 4.2]{BK3} that $\tilde{\mu} \in \mathcal{O}_{\phi_{n, \omega}}(\ucpan)$, and therefore, the map $k_{\tilde{\mu}} : \text{L}^\infty (\ucpan, \tilde{\mu}) \rightarrow \rho(A)^\prime$ is a *-isomorphism onto its range (refer Proposition \ref{prop;om}). Similarly, as we have shown in the previous example, we show that, if the algebra $k_\mu(\text{L}^\infty (S(A), \mu))$ is contained in $\{ \pi_{\omega}(A) \cup U_\omega(G) \}^\prime$, then $k_{\tilde{\mu}}( \text{L}^\infty (\ucpan, \tilde{\mu}))$ is contained in $\{ \rho(A) \cup U_\phi (G) \}^\prime$. But from \cite[Example 4.2]{BK3} we have 
\begin{equation*}
k_{\tilde{\mu}}(f) = \begin{bmatrix}
k_\mu(f\circ \tilde{f})       &0 &\ldots  &0\\
0 & k_\mu(f\circ \tilde{f})      &\ddots  &\vdots\\
\vdots  &\ddots  &k_\mu(f\circ \tilde{f})       &0\\
0 &\ldots  &0 &k_\mu(f\circ \tilde{f})
\end{bmatrix}_{n \times n}
\end{equation*}
and
\begin{equation*}
\rho(a) = \begin{bmatrix}
\pi_\omega(a)      &0 &\ldots  &0\\
0 & \pi_\omega(a)     &\ddots  &\vdots\\
\vdots  &\ddots  &\pi_\omega(a)      &0\\
0 &\ldots  &0 &\pi_\omega(a)
\end{bmatrix}_{n \times n}.
\end{equation*}
We have $\omega \in S_G(A)$ and therefore $\phi_{n,\omega} \in \ucpang$. This implies for fixed $g \in G$ we get 
\begin{equation*}
U_\omega(g) \pi_{\omega}(a) U_\omega(g)^* = \pi_{\omega}(\tau_g(a)) \hspace{10pt} \text{and} \hspace{10pt}
U_\phi(g) \rho(a) U_\phi(g)^* = \rho(\tau_g(a)).
\end{equation*}
From this, we identify $U_\phi(g)$ with
\begin{equation*}
\begin{bmatrix}
U_\omega(g)      &0 &\ldots  &0\\
0 & U_\omega(g)     &\ddots  &\vdots\\
\vdots  &\ddots  &U_\omega(g)       &0\\
0 &\ldots  &0 &U_\omega(g)
\end{bmatrix}_{n \times n}.
\end{equation*}
	
We have assumed that $k_\mu(\text{L}^\infty (S(A), \mu))$ is contained in $\{ \pi_{\omega}(A) \cup U_\omega(G) \}^\prime$. This implies $k_\mu(\text{L}^\infty (S(A), \mu)) \subseteq U_\omega(G)^\prime$. Then using the identification of $k_{\tilde{\mu}}(f)$ for $f \in \text{L}^\infty (\ucpan, \tilde{\mu}))$, and $U_\phi(g)$ as above, we get $k_{\tilde{\mu}}(f)$ belongs to $U_\phi (G)^\prime$ for all $f \in \text{L}^\infty (\ucpan, \tilde{\mu}))$. This shows that $k_{\tilde{\mu}}( \text{L}^\infty (\ucpan, \tilde{\mu}))$ is contained in $\{ \rho(A) \cup U_\phi (G) \}^\prime$. Again, as an application of Theorem \ref{thm;inv} and Proposition \ref{prop;supp} we get a $G$-invariant decomposition of $\phi_{n, \omega}$.
\end{example}

\subsection{Open Questions}
We conclude this article with a few questions. This section explores examples of $G$-invariant UCP maps admitting decomposition into $G$-invariant UCP maps using orthogonal measures. In Example \ref{ex;Ginv}, we discussed the map $\phi_G$, which is a $G$-invariant UCP map. Further study of the map $\phi_G$ would be interesting. In connection with Example \ref{ex;Ginv}, we pose the following two questions 
\begin{enumerate}
\item Can we characterize the cases when the map $\phi_G$ is a $G$-ergodic UCP map?
\item Does the map $\phi_G$ admit a $G$ -invariant decomposition?
\end{enumerate}


\section*{Acknowledgement} 
The first named author is partially supported by Science and Engineering Board (DST, Govt. of India) grant no. ECR/2018/000016 and the second named author is supported by CSIR PhD scholarship award letter no. 09/1020(0142)/2019-EMR-I.

The authors would like to express their sincere gratitude to the referee for offering constructive suggestions on the initial draft that helped significantly to improve the manuscript.

\vspace{20pt}
\subsection*{Declaration}
The authors declare that there is no conflict of interest. 

\bibliographystyle{plain}
\bibliography{mybib}

\end{document}